\documentclass[12pt]{amsart}
\usepackage{amsmath}
\usepackage{amsfonts}
\usepackage{amssymb}
\usepackage{amscd}
\usepackage{fancyvrb}
\usepackage{bm}
\usepackage{xypic}
\usepackage{tikz}

\usepackage{graphicx}
\usepackage{epsfig}
\usepackage{pstricks}

\newtheorem{theorem}{Theorem}

\newtheorem{lemma}[theorem]{Lemma}

\newtheorem{Conjecture}[theorem]{Conjecture}

\newcommand{\dbZ}{\mathbb{Z}}

\newcommand{\calK}{{\mathcal K}}

\newcommand{\dbH}{{\mathbb H}}

\newcommand{\func}[3]{#1:#2\to#3}

\begin{document}

\title[Ranks of the algebraic $K$-Theory of hyperbolic groups]
{\bf On the ranks of the algebraic $K$-Theory of hyperbolic groups}

\author{Daniel Juan-Pineda}
\address{Centro de Ciencias Matem\'aticas,
Universidad Nacional Aut\'onoma de M\'exico, Campus Morelia, Apartado Postal 61-3 (Xangari), Morelia,
 Mi\-cho\-a\-c\'an, MEXICO 58089}
\email{daniel@matmor.unam.mx}
\thanks{We acknolwedge support from research grants from DGAPA-UNAM and CONACyT-M\'exico}

\author{Luis Jorge S\'anchez Salda\~na}
\address{Centro de Ciencias Matem\'aticas, Universidad Nacional Aut\'onoma de M\'exico, Apartado Postal 61-3 (Xangari), Morelia,
 Mi\-cho\-a\-c\'an, MEXICO 58089}

\email{luisjorge@matmor.unam.mx}
\dedicatory{Dedicated to Fico on the occassion of his 70th birthday.}


\begin{abstract}
Let $G$ be a word hyperbolic group. We prove that the algebraic
$K$-theory groups of $\dbZ [G]$, $K_n(\dbZ[G])$, have finite rank
for all $n\in \dbZ$. For a few classes of groups, we give explicit
formulas for the ranks of the algebraic $K$-theory groups of their group rings.

\end{abstract}

\maketitle

\section{Introduction and preliminaries}

Recall that the Farrell-Jones isomorphism conjecture proposes
that, for any discrete group $G$,  the algebraic $K$-theory of the
group ring $\dbZ [G]$ is determined by the algebraic $K$-theory of
the virtually cyclic subgroups of $G$ plus homological
information.

\begin{Conjecture}
\textbf{(Farrell-Jones Isomorphism Conjecture, IC)} Let $G$ be a discrete group. Then for all $n\in \mathbb{Z}$ the assembly map
\begin{equation}
\func{A_{Vcyc}}{H_n^G(\underline{\underline{E}}G;\calK)}{H_n^G(pt;\calK)\cong K_n(\mathbb{Z}[G])}
\end{equation}
 induced by the projection $\underline{\underline{E}}G\rightarrow pt$ is an isomorphism,
 where $H^G_*(-;\calK)$ is a suitable equivariant homology theory with local coefficients in $\calK$, the non-connective spectrum of algebraic $K$ theory and $\underline{\underline{E}}G$ is a model for the classifying space for actions with isotropy in the family of virtually cyclic subgroups of $G$.
\end{Conjecture}

This conjecture has been verified, among others, when $G$ is a
word  hyperbolic group \cite{hyperbolic}, or a $CAT(0)$-group
\cite{wegner}. Once we know the conjecture holds for a group $G$,
we can try to compute $K_n(\mathbb{Z}[G])\cong
H_n^G(\underline{\underline{E}}G)$ using an Atiyah-Hirzebruch type
spectral sequence.

In this paper we use the validity of the  Farrell-Jones conjecture
and the corresponding spectral sequence to show that the rank of
$K_n(\dbZ [G])$ is finite for all $n\in \dbZ$, where $G$ is a word
hyperbolic group. Next, we give some explicit examples of
computations of these ranks.

For hyperbolic groups Leary and Juan-Pineda \cite{jefe}, showed that
\begin{equation}
H_n^G(\underline{\underline{E}}G;\calK) \cong  H_n^G(\underline{E}G;\calK) \oplus  \bigoplus_{(V)} cok_n(V),
\end{equation}
where $\underline{E}G$ is the classifying space for the family \textit{FIN}, of finite subgroups of $G$, $(V)$ consists of one representative from each conjugacy class of maximal infinite virtually cyclic subgroup of $G$ and $cok_n(V)$ is the cokernel of the homomorphism  $H_n^V(\underline{E}V \rightarrow pt;\calK).$\\

 It is well known, see \cite[Thm. 5.9 and page  4.]{torsionnils} that the terms $cok_n{V}$ are torsion groups. This gives the following:

 \begin{lemma} \label{hyprank}
 Let $G$ be a discrete word hyperbolic group. Then for all $n\in \dbZ$
 $$rank (K_n(\dbZ [G]))= rank( H_n^G(\underline{E}G;\calK)).$$

 \end{lemma}

This paper is, in part, complementary to \cite{vfree} where we treated the case  \textsl{lower} $K$ groups, namely $K_i()$ for $i\leq 1$. Here we treat the whole spectrum of $K$ theory and a broader class of groups.


\section{Ranks.}

In view of Lemma \ref{hyprank} the rank of the algebraic
$K$-theory groups of $\dbZ [G]$ are  determined by  the ranks of
the algebraic $K$-theory of the finite subgroups of $G$ and the
homology of $\underline{E}G$. The ranks of the algebraic $K$
groups of the group ring of a finite group are given as follows:

\begin{theorem}\label{finiterank}(\cite{jahren}, \cite{rank1}, \cite{carter1}, \cite{carter2})
Let $H$ be a finite group with $r$ distinct real irreducible
representations, $c$ of them of complex type,  and $q$ distinct
rational irreducible representations. For $n>1$ we then have

\begin{equation*}
rank(K_n(\dbZ [H])) =
\begin{cases}
r & if\ n\equiv 1\ mod\ 4,\\
c & if\ n\equiv 3\ mod\ 4,\\
0 & if\ n\ is\ even.\end{cases}
\end{equation*}
When $n\leq 1$, we have:
\begin{equation*}
rank(K_1(\dbZ [H])) = r-q,
\end{equation*}

\begin{equation*}
rank(K_0(\dbZ [H])) = 1,
\end{equation*}

\begin{equation*}
rank(K_{-1}(\dbZ [H]))< \infty \text{ and}
\end{equation*}

\begin{equation*}
K_{-n}(\dbZ [H])=0\ n>1.
\end{equation*}

\end{theorem}

Note that $r$ is equal to the number of real conjugacy classes of $H$, that is, classes of the form $C(h)=\{ghg^{-1}, gh^{-1}g^{-1}| g\in H\}$, $c$ is equal to the number of real conjugacy classes such that $C(h)\neq \{hgh^{-1} | h\in H \}$, and $q$ is the number of conjugacy classes of cyclic subgroups of $H$, see \cite{serre}. \\

To compute the equivariant homology groups
$H_*^G(\underline{E}G;\calK)$ we may use an Atiyah-Hirzebruch type
spectral  sequence. Let  $C_n$ denote the set of $n$-cells of the
space $\underline{B}G=\underline{E}G / G$, then the first page of
our spectral sequence is given by
$$
\xymatrix{  &  \vdots& \vdots &  \\
\cdots &\displaystyle \bigoplus_{\sigma^{p}\in C_{n}} K_q(\dbZ [G_{\sigma^{p}}])\ar[l] &  \displaystyle\bigoplus_{\sigma^{p+1}\in C_{n+1}} K_q(\dbZ [G_{\sigma^{p+1}}])\ar[l] & \cdots\ar[l]
\\ \cdots & \displaystyle \bigoplus_{\sigma^{p}\in C_{n}} K_{q-1}(\dbZ [G_{\sigma^{p}}])\ar[l] &  \displaystyle\bigoplus_{\sigma^{p+1}\in C_{n+1}} K_{q-1}(\dbZ [G_{\sigma^{p+1}}])\ar[l] & \cdots \ar[l] \\  &  \vdots & \vdots &  }
$$
where $G_{\sigma}$ denotes the stabilizer of a pre-image $\sigma
'\in \underline{E}G$ of $\sigma \in \underline{B}G$, and  the
homomorphisms in the chain complex are induced by the natural
inclusions (up to conjugacy). Note that $G_{\sigma}$ is always a
finite group, hence we can apply Theorem \ref{finiterank} to every
group appearing in our spectral sequence. As a consequence we may
identify the second page of the Atiyah-Hirzebruch spectra sequence
as
$$
E^2_{p,q}= \dbH_2(\underline{B}G;\{\calK_q({\dbZ [G_\sigma]})\}),
$$
where the above is a homology theory with local coefficients given
by the algebraic $K$ groups  of the group rings $\dbZ[G_\sigma]$
for all the finite isotropy groups $G_\sigma$.

\begin{theorem}\label{main}
Let $G$ be an hyperbolic group. Then $rank(K_n(\dbZ [G]))$ is
finite for all $n\in \dbZ$.
\end{theorem}
\begin{proof}
It is known that for  $G$ word hyperbolic, there exists a finite
model  for $\underline{E}G$, i.e., such that $\underline{B}G$ is
compact, see for example \cite{eghyperbolic}. Take  this finite
model for $\underline{E}G$. Hence the only possible  non-zero
terms in the $n$th page of our spectral sequence $E^n_{p,q}$ are
those terms with $0\leq p \leq m$, $m=dim \underline{B}G$, that
is, they are contained in a vertical strip for all $n \in
\{1,2,3,...\} \cup \{ \infty \}$. Now, since $E^{\infty}_{p,q}$
has finite rank because it is the quotient of a subgroup of the
abelian group $E^1_{p,q}$ and
$$
rank(K_n(\dbZ [G]))= \sum_{p+q=n} rank(E^{\infty}_{p,q})
$$
the proof follows by Theorem \ref{finiterank} and the compactness
of $\underline{B}G$.
\end{proof}


Note that using \cite[Thm. 5.11]{torsionnils} Lemma \ref{hyprank}
is valid for every group satisfying the Farrell-Jones conjecture,
hence following the proof of Theorem \ref{main} we have:


\begin{theorem}
Let $G$ be a group that admits a finite model for $\underline{B}G$ and such that satisfies the Farrell-Jones conjecture. Then $rank(K_n(\dbZ [G]))$ is finite for all $n\in \dbZ$.
\end{theorem}

This last Theorem is more general and applies, for instance, to  the groups that appear in
\cite{LO09}.


\section{Examples.}

In this section we give some explicit computations of $rank(K_n(\dbZ [G]))$.
\subsection{Finitely generated free groups.}

Let $F_n$ be the free group on $n$ generators, $n\in \dbZ$. Since $F_n$ is torsion free $\underline{E}G=EG$, on the other hand we know that the Cayley graph of $G$ is a model for $EG$, and $BG$ with this model is a wedge of circles. Hence there is one $0$-cell  and $n$ $1$-cells. Moreover, $G_\sigma=1$ for all cells, hence
$$
E^2_{pq}=H_p(\vee_nS^1;\{K_q\})=H_p(\vee_nS^1;K_q(\dbZ)).
$$
This gives
$$
H_p(\underline{B}G;K_q(\dbZ))=\begin{cases}
   K_q(\dbZ)& \text{ for } p=0,\\
   \oplus_n(K_q(\dbZ))& \text{ for } p=1,\\
   0& \text{ for } p>1 \text{ or } q\leq -1.
   \end{cases}
$$

The graph associated to the free group on $n$ generators, the labels are the coefficients of the corresponding cell, these have all trivial
stabilizers.
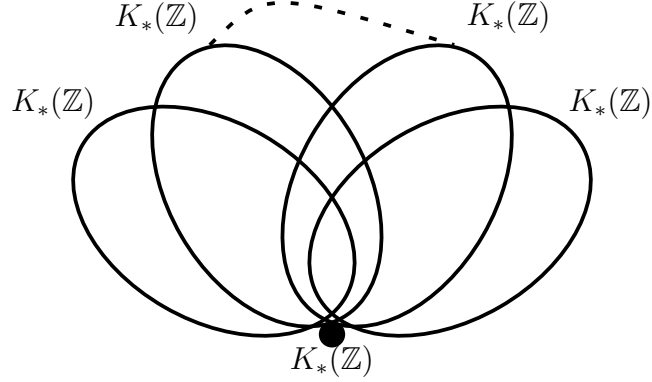
\begin{center}
\begin{tikzpicture}[scale=2.5]
\draw[line width=.5mm,rotate=60] (-.65,.35) ellipse (15pt and 23pt);
\draw[line width=.5mm,rotate=-60] (.65,.35) ellipse (15pt and 23pt);
\draw (-1.2,.10) node[above left] {$K_*(\mathbb{Z})$};
\draw (1.2,.10) node[above right] {$K_*(\mathbb{Z})$};
\draw[line width=.5mm,rotate=30] (-.4,0) ellipse (15pt and 23pt);
\draw (-.65,.55) node[above left] {$K_*(\mathbb{Z})$};
\draw[line width=.5mm,rotate=-30] (.4,0) ellipse (15pt and 23pt);
\draw[line width=.5mm, loosely dashed] (-.65,.55).. controls (-.35,.85).. (.65,.55) node[above right] {$K_*(\mathbb{Z})$};
\fill (0,-.99) circle (.7mm) node[below ] {$K_*(\mathbb{Z})$};
\end{tikzpicture}
\end{center}

Notice that all the differentials vanish, hence this spectral sequence collapses at this stage giving
$$
rank(K_n(\dbZ [F_n]))= rank(K_n(\dbZ))+n\cdot rank(K_{n-1}(\dbZ)).
$$
Applying Theorem \ref{finiterank} to the trivial group we have that
$$
rank(K_n(\dbZ)) =
\begin{cases}
1 & if\ n\equiv 1\ mod\ 4\ \text{ and }\ n>1,\ or\ n=0,\\
0 & \text{ otherwise }.\\
\end{cases}
$$
it follows that
$$
rank(K_n(\dbZ [F_n])) =
\begin{cases}
1 & if\ n\equiv 1\ mod\ 4\ \text{ and }\ n>1,\ or\ n=0,\\
n & if\ n\equiv 2\ mod\ 4\ \text{ and }\ n>2,\ or\ n=1,\\
0 & \text{ otherwise }.
 \end{cases}
$$
\subsection{Free products of finite groups}

Let $G_1$ and $G_2$ be finite groups, and let  $G=G_1*G_2$ be their free product. We can find a one-dimensional model for $\underline{E}G$ such that $\underline{B}G$ is a closed interval with trivial isotropy in the edge and isotropy at the vertices $G_1$ and $G_2$ (\cite{trees}:
\begin{center}
\begin{tikzpicture}[scale=2.5]
 \fill (-1,0) circle (.7mm) node[above left] {$K_*(\mathbb{Z}[G_1])$};
\draw[line width=.5mm] (-1,0)-- (1,0);
\fill (1,0) circle (.7mm) node[above right ] {$K_*(\mathbb{Z}[G_2])$};
\node (0,0) [above] {$K_*(\mathbb{Z})$};
\end{tikzpicture}
\end{center}

 Once again our spectral sequence collapses at the second page giving
$$
rank(K_{n}(\dbZ [G]))=rank (K_{n}(\dbZ [G_1]))+ rank (K_{n}(\dbZ [G_2])) - rank (K_{n-1}(\dbZ ))
$$ and
$$
rank(K_n(\dbZ [G]))=
\begin{cases}
1 & n=0,\\
r_1+r_2-q_1-q_2 & i=1,\\
r_1+r_2-1 & if\ n\equiv 1\ mod\ 4,\ n>1,\\
c_1+c_2 & if\ n\equiv 3\ mod\ 4,\ n>1,\\
0 & \text{ otherwise }.
\end{cases}
$$
where $r_i$ is the number of distinct real irreducible representations of $G_i$, $i=1,2$; $c_i$ is the number of distinct real irreducible representations of complex type of $G_i$, $i=1,2$; and $q_i$ is the number of distinct rational irreducible representations of $G_i$, $i=1,2$.


\subsection{\bm{$PSL_2(\dbZ)$}}

This is a particular case of the previous example since $PSL_2(\dbZ)\cong \dbZ_2 * \dbZ_3$. Using the notation from above, we set $G_1=\dbZ [\dbZ_2]$ and $G_2=\dbZ[\dbZ_3]$, hence $r_1=2$, $r_2=2$, $c_1=0$, $c_2=1$, $q_1=2$, $q_2=2$ and
$$
rank(K_n(\dbZ [PSL_2(\dbZ)]))=
\begin{cases}
0& \text{ if }\ n=-1,\\
1 & n=0,\\
0 & n=1,\\
3 & \text{ if }\ n\equiv 1\ mod\ 4,\ n>1,\\
1 & \text{ if }\ n\equiv 3\ mod\ 4,\ n>1,\\
0 & \text{ otherwise }.\\
\end{cases}
$$

\subsection{The fundamental group of a closed orientable aspherical surface.}
Let $S_g$ be the orientable closed surface of genus $g>1$. Since the universal covering of $S_g$ is contractible we have that $S_g$ is a model for $B\pi_1(S_g)$. Furthermore,
 $S_g$ is a model for $\underline{B}\pi_1(S_g)$ as well. Moreover, these groups are hyperbolic as $S_g$ is a compact surface that admits a metric with constant curvature $-1$.
Using the classical construction of $S_g$ as the quotient of a $4g$-agon we can give $S_g$ a CW-structure consisting of one $0$-cell, $2g$ 1-cells, and one 2-cell and they all have trivial isotropy. Hence the second term of the  Atiyah-Hirzebruch spectral sequence has constant coefficients the $K$-theory of the integers:
$$
H_p(S_g; K_q(\dbZ))=\begin{cases}
K_q(\dbZ) \text{ for } p=0,\\
\displaystyle\bigoplus_{2g} K_q(\dbZ)& \text{ for } p=1,\\
K_q(\dbZ)& \text{ for } p=2,\\
0& \text{ for } p>1 \text{ or } q<0.
\end{cases}
$$
Once again all differentials are trivial and our spectral sequence collapses. This gives
$$
rank(K_n(\dbZ [\pi_1(S_g)])=
\begin{cases}
1 & n=0,2\ \text{ or } \ n\equiv 1,3\ mod\ 4\ n>1,\\
2g & \text{ if }\ n=1\ \text{ or }\ n\equiv 2\ mod\ 4,\ n>1,\\
0 & \text{ otherwise }.\\
\end{cases}
$$


\subsection{Finitely generated virtually free groups}
Let $G$ be a finitely generated virtually free group,that is $G$ has a finitely generated free subgroup of finite index. As finitely generated
free groups are hyperbolic and a group with a $\delta$-hyperbolic subgroup of finte index is hyperbolic, it follows that $G$ is $\delta$-hyperbolic as well.
 By Bass-Serre theory
\cite{trees} and by the work in \cite{vfree}, one can find a tree $T$ on which $G$ acts with finite stabilizers, that is $T$ is a model for $\underline{E}G$. Let $\mathbf{E}$ and $\mathbf{V}$ be the set of edges and vertices of the graph of groups for $G$ determined by  $T$, this part is developed in \cite{vfree}[section2.2]. In order to calculate the second page of our spectral sequence, observe that $T$ has only cells of dimension $0$ and $1$ hence our page is
$$
E^2_{p,q}=\begin{cases}
         coker\left (\displaystyle\bigoplus_{e\in \mathbf{E}}K_q(\dbZ[G_e])\to \bigoplus_{v\in\mathbf{V}}K_q(\dbZ[G_v])\right ),& \text{ for } p=0,\\
         ker\left (\displaystyle\bigoplus_{e\in \mathbf{E}}K_q(\dbZ[G_e])\to \bigoplus_{v\in\mathbf{V}}K_q(\dbZ[G_v])\right ),& \text{ for } p=1,\\
         0& \text{ otherwise }.
\end{cases}
$$
It follows that the differentials vanish and our spectral sequence collapses at this page hence
\begin{eqnarray*}
H_n^G(\underline{E}G; \{\calK_q\})=&coker\left (\displaystyle\bigoplus_{e\in \mathbf{E}}K_n(\dbZ[G_e])\to \bigoplus_{v\in\mathbf{V}}K_n(\dbZ[G_v])\right )\\
&\bigoplus\quad
         ker\left (\displaystyle\bigoplus_{e\in \mathbf{E}}K_{n-1}(\dbZ[G_e])\to \bigoplus_{v\in\mathbf{V}}K_{n-1}(\dbZ[G_v])\right).
         \end{eqnarray*}
In order to simplify the notation, let us define
\begin{eqnarray*}
E_n=& \displaystyle\bigoplus_{e\in \mathbf{E}}K_n(\dbZ[G_e]),\\
V_n=&\displaystyle\bigoplus_{v\in\mathbf{V}}K_n(\dbZ[G_v]),\\
Ker_n=&ker(E_n\to V_n) \text{ and }\\
Cok_n=&coker(E_n\to V_n).
\end{eqnarray*}
In this way we have that
$$
H_n^G(\underline{E}G; \{\calK_q\})=Cok_n\oplus Ker_{n-1}.
$$
These last groups depend on the graph structure of our tree with
the stabilizers of the action, which are all finite.

\subsection{\bm{$G=F_n\rtimes S_n$}}
This example is worked out in detail in \cite{vfree}[\S 3] for
other purposes. Let $G=F_n\rtimes S_n$ with the symmetric group
$S_n$ on $n$ letters, acting on the free group on $n$ generators
by permuting the generators. The graph of groups is a single loop
with vertex group $S_{n-1}$ and edge group $S_n$. In this case the
morphisms
$$
E_i\to V_i
$$
are all zero, it follows that
$$
H_i^G(\underline{E}G; \{\calK_q\})=K_i(\dbZ[S_n])\oplus K_{i-1}(\dbZ[S_{n-1}]).
$$

It is well known that the conjugation class of an element $x\in
S_n$ is determined by its cyclic decomposition, since $x$ and
$x^{-1}$ have the same cyclic decomposition we have that they
belong to the same conjugacy class. Hence the number of real
conjugacy classes of $S_n$ is equal to $p(n)$, the number of
partitions of $n$, and the number of real conjugacy classes of
complex type is zero. Finally if two elements on $S_n$ determine
the same cyclic subgroup then they are conjugate, this implies
that the number of conjugacy classes of cyclic subgroups of $S_n$
is equal to $p(n)$.

$$
rank(K_i(\dbZ S_n))=
\begin{cases}
p(n) & i\equiv 1\ mod\ 4\ i>1,\\
1 & i=0,\\
0 & \text{ otherwise, } \\
\end{cases}
$$

and

$$
rank(K_i(\dbZ G))=
\begin{cases}
p(n) & i\equiv 1\ mod\ 4\ i>1,\\
p(n-1) & i\equiv 2\ mod\ 4\ i>1,\\
1 & i=0,\ 1,\\
0 & \text{ otherwise. } \\
\end{cases}
$$

\bibliographystyle{alpha}
\bibliography{myblib}
\end{document}